\documentclass[11pt,a4paper]{article}

\usepackage{amsmath}
\usepackage{amsthm}
\usepackage{amsfonts}
\usepackage{amssymb}

\usepackage{authblk}

\newtheorem{thm}{Theorem}[section]
\newtheorem{lem}[thm]{Lemma}

\newtheorem{cor}[thm]{Corollary}

\newtheorem{df}[thm]{Definition}

\usepackage[toc,page]{appendix}

\title{E-polynomials and Terwilliger Algebras Related to Codes}
\author{Nur Hamid}
\affil{Universitas Nurul Jadid}
\begin{document}

\maketitle

\begin{abstract}
This paper investigates the Terwilliger algebra of some group association schemes related to codes. 
In addition, it also shows the generators of invariant rings appearing by E-polynomials. 
\end{abstract}

\section{Introduction}

The notion of Terwilliger algebra was introduced by Paul Terwilliger in \cite{terwilliger}. 
After the introduction,
there were some works related to it. 
Other works discussing Terwilliger algebra can be found in \cite{fernadezmiklavic,HamidTerwilliger}. 
For the recent paper,
the reader can see \cite{balmacedareyes,hanaki}.
The computation is done using \cite{sagemath}.

The main objective of this paper is to obtain the Terwilliger algebras of some group association schemes for the the groups $G_I$, $G_{II}, G_{III},G_{IV}$.
Let $X \in \lbrace \text{I}, \text{ II}, \text{ III}, \text{ IV} \rbrace$.
It is known that a weight enumerator of Type $X$ codes is $G_X$-invariant.
The reader who is interested in Terwilliger algebra can directly see Section \ref{sec: Terwilliger algebra}.
The section shows that
\begin{align*}
T(G_{I}) & \cong \mathcal{M}_1 \oplus \mathcal{M}_1 \oplus \mathcal{M}_2 \oplus \mathcal{M}_3 \oplus \mathcal{M}_7, \\
T(G_{II}) & \cong \mathcal{M}_4 \oplus \mathcal{M}_8 \oplus \mathcal{M}_{12} \oplus \mathcal{M}_{16} \oplus \mathcal{M}_{24} \oplus \mathcal{M}_{32}, \\
T(G_{III}) & \cong \mathcal{M}_2 \oplus \mathcal{M}_{10} \oplus \mathcal{M}_{16} \\
T(G_{IV}) & \cong \mathcal{M}_2 \oplus \mathcal{M}_2 \oplus \mathcal{M}_{6}.
\end{align*}

From coding theory point of view,
the invariant theory of finite groups can connect number theory to coding theory.
The point of view continues to show that the ring of the weight enumerators of Type $X$ codes can be generated by Eisenstein polynomials (E-polynomials for short) associated to Type $X$ codes. 
In other words, 
it is obtained that 
\[
\mathfrak{R}^{G_{I}} = \mathbb{C}[\varphi_2, \varphi_8]
\]
\[
\mathfrak{R}^{G_{II}} = \mathbb{C}[\varphi_8, \varphi_{24}]
\]
\[
\mathfrak{R}^{G_{III}} = \mathbb{C}[\varphi_4, \varphi_{12}]
\]
\[
\mathfrak{R}^{G_{IV}} = \mathbb{C}[\varphi_2, \varphi_6]
\]
where $\mathfrak{R}^{G_X}$ denotes the invariant ring for $G_X$.
Some results in Section \ref{sec:E-polynomials} are not new.
For example, the E-polynomials related to $G_{\text{II}}$ was discussed in \cite{ourapoly}.

\section{Preliminaries}

We follow \cite{MiezakiOura} for the notations. 
Let 
\[
G_I = \left\langle \frac{1}{\sqrt{2}} 
\begin{pmatrix}
1 & 1 \\
1 & -1
\end{pmatrix},
\begin{pmatrix}
1 & 0 \\
0 & -1
\end{pmatrix}
\right\rangle,
\]
\[
G_{II} = \left\langle \frac{1}{\sqrt{2}} 
\begin{pmatrix}
1 & 1 \\
1 & -1
\end{pmatrix},
\begin{pmatrix}
1 & 0 \\
0 & i
\end{pmatrix}
\right\rangle,
\]
\[
G_{III} = \left\langle \frac{1}{\sqrt{3}} 
\begin{pmatrix}
1 & 2 \\
1 & -1	
\end{pmatrix},
\begin{pmatrix}
1 & 0 \\
0 & \exp(2 \pi i / 3)
\end{pmatrix}
\right\rangle,
\]
\[
G_{IV} = \left\langle \frac{1}{2} 
\begin{pmatrix}
1 & 3 \\
1 & -1
\end{pmatrix},
\begin{pmatrix}
1 & 0 \\
0 & -1
\end{pmatrix}
\right\rangle.
\]
The orders and the number of conjugacy classes for these groups can be seen in Table \ref{tab:order of G}.

\begin{table}
\centering
\begin{tabular}{cccc}
 $G$ & $|G|$ & Conjugacy Classes \\
 \hline
 $G_I$ & 16 & 7 \\
 $G_{II}$ & 192 & 32 \\
 $G_{III}$ & 48 & 14 \\
 $G_{IV}$ & 12 & 6
\end{tabular}
\caption{Order and conjugacy classes of some groups}
\label{tab:order of G}
\end{table}

Now some terms in coding theory are given. 
Let $\mathbb{F}_q$ be the field of $q$ elements. 
A linear code $C$ of length $n$ is a linear subspace of $\mathbb{F}_q^n$.
The number of nonzero components of $\mathbf{c}$ is called the weight of $\mathbf{c}$ and denoted by $wt(\mathbf{c})$.
The weight enumerator $w_C (x,y)$ of a code $C$ is defined by
\[
w_C (x,y) := \sum_{\mathbf{c} \in C} x^{n- wt(\mathbf{c})} y^{wt(\mathbf{c})}.
\]
The inner product $(\mathbf{x}, \mathbf{y})$ of two elements $\mathbf{x}, \mathbf{y} \in \mathbb{F}_q^n$ is defined by
\[
(\mathbf{x}, \mathbf{y}) := x_1 y_1 + \cdots + x_n y_n \text{  mod  } n.
\]
The dual of $C$, denoted by $C^{\perp}$, is defined by
\[
C^{\perp} := \lbrace \mathbf{y} \in \mathbb{F}_q^n \ \vert \ (\mathbf{x},\mathbf{y})=0, \forall \mathbf{x} \in C \rbrace.
\]
We say $C$ self-dual if $C=C^\perp$ holds.

Let $C$ be a self-dual code. 
The self-dual code is said to be as follows:
\begin{enumerate}
\item Type I if it is defined over $\mathbb{F}_2^q$ with all weights multiples of of 2;
\item Type II if it is defined over $\mathbb{F}_2^n$ with all weights multiples of 4; 
\item Type III if it is defined over $\mathbb{F}_3^n$ with all weights multiples of 3; and 
\item Type IV if it is defined over $\mathbb{F}_4^n$ with all weights multiples of 2. 
\end{enumerate} 
It is necessary to note that the weight enumerator $w_C(x,y)$ for Type $X$ codes is in invariant ring of $G_X$ for $X = \lbrace \text{I}, \text{ II}, \text{ III}, \text{ IV} \rbrace$.

Let $\mathfrak{R}^{G_X}$ be the invariant ring of $G_X$. That is
\[
\mathfrak{R}^G = \lbrace f \in  \mathbb{C}[x,y] 
\ \vert \ f^g = f, \ \forall g \in G_X \rbrace.
\]
Here $f^g$ means the action of $g$ on $f$.
In this paper,
the dimension formula $\mathcal{I}(G_X)$ for $\mathfrak{R}^{G_X}$ is written by the formal series
\[
\mathcal{I}(G_X) = \sum_k ^\infty \dim \mathfrak{R}^{G_X}_k t^k.
\]
The dimension formulas for $\mathfrak{R}^G$ where $G=G_I, G_{II},G_{III},G_{IV}$ are:
\[
\mathcal{I}(G_{I}) = \frac{1}{(1-t^2)(1-t^8)},
\]
\[
\mathcal{I}(G_{II}) = \frac{1}{(1-t^8)(1-t^{24})},
\]
\[
\mathcal{I}(G_{III}) = \frac{1}{(1-t^4)(1-t^{12})},
\]
\[
\mathcal{I}(G_{IV}) = \frac{1}{(1-t^2)(1-t^6)}.
\]

\section{E-polynomials} \label{sec:E-polynomials}

This section discusses E-polynomials. 
Before proceeding, 
some information about the generators of some invariants rings mentioned before is provided. 
It refer to \cite{MiezakiOura} for the generators. 
\begin{enumerate}
\item $\mathfrak{R}^{G_I} = \mathbb{C}[f,g], f=x^2+y^2, g=x^2 y^2(x^2-y^2)^2$

\item $\mathfrak{R}^{G_{II}} = \mathbb{C}[f,g], f=x^8 + 14 x^4 y^4 + y^8, g=x^4 y^4(x^4-y^4)^4$

\item $\mathfrak{R}^{G_{III}} = \mathbb{C}[f,g], f=x^4+8xy^3, g=y^3(x^3-y^3)^3$

\item $\mathfrak{R}^{G_{IV}} = \mathbb{C}[f,g], f=x^2 + 3y^2, g=y^2(x^2-y^2)^2$
\end{enumerate}
Hence, $R^\mathbb{G_X}$ can be generated by the E-polynomials related to Type $X$ codes. 

Let $\bar{x}$ be a column vector of $x$ and $y$.
An E-polynomial $\varphi_k$ of degree $k$ for the group $G_X$ with respect to Type $X$ code is defined by
\[
\varphi_k (\bar{x})=\varphi_k(x,y)= \frac{1}{|G|} \sum_{\sigma \in G} (\sigma_1 \bar{x})^k
\]
where $\sigma_1$ is the first row of $\sigma$.
It is not difficult to show that E-polynomial for $G_X$ belongs to $\mathbb{C}[x,y]^{G_X}$.

Let $\mathfrak{E}^{G_X}$ be the ring of E-polynomials for the group $G_{X}$.
By obtaining the generators of $\mathfrak{E}^{G_X}$,
the following theorem is obtained.
\begin{thm} \label{thm:epoly generated}
\[
\mathfrak{E}^{G_I} = \mathbb{C}[\varphi_2, \varphi_8]
\]
\[
\mathfrak{E}^{G_{II}} = \mathbb{C}[\varphi_8, \varphi_{24}]
\]
\[
\mathfrak{E}^{G_{II}} = \mathbb{C}[\varphi_4, \varphi_{12}]
\]
\[
\mathfrak{E}^{G_{IV}} = \mathbb{C}[\varphi_2, \varphi_{6}]
\]
\end{thm}
\begin{proof}
This is done by computation.
The proof is similar to \cite[Theorem 4.2.]{Hamid}.
\end{proof}

Finding the generators of $\mathfrak{E}^{G_X}$, 
I observe if the generators can generated the invariant ring $\mathfrak{R}^{G_X}$.
The following theorem shows that the invariant ring of the group $G_I, G_{II}, G_{III}, G_{IV}$ can be generated by the E-polynomials for each group. 
\begin{thm} \label{thm:invariant generator}
The followings hold
\[
\mathbb{R}^{G_{I}} = \mathbb{C}[\varphi_2, \varphi_8]
\]
\[
\mathbb{R}^{G_{II}} = \mathbb{C}[\varphi_8, \varphi_{24}]
\]
\[
\mathbb{R}^{G_{III}} = \mathbb{C}[\varphi_4, \varphi_{12}]
\]
\[
\mathbb{R}^{G_{IV}} = \mathbb{C}[\varphi_2, \varphi_6]
\]
\end{thm}

\begin{proof}
We give the proof for $G_{III}$ case. 
The proofs of $G_{I}, G_{II},$ and $G_{IV}$ cases are similar. 
The explicit form of $\varphi_4$ and $\varphi_{12}$ are
\begin{align*}
\varphi_4 & = \frac{1}{3} \left(x^{4} + 8 x y^{3} \right), \\
\varphi_{12} & = 243 \left( 61 x^{12} + 440 x^{9} y^{3} + 14784 x^{6} y^{6} + 28160 x^{3} y^{9} + 1024 y^{12} \right).
\end{align*}
Then $f$ and $g$ for $G_{III}$ case can be expressed as
\begin{align*}
f = & 3 \varphi_4, \\
g = & \frac{1}{1024} \left(1647 \varphi_4^3 - 243 \varphi_{12} \right).
\end{align*}
In the same way,
other cases could be proved. 
This completes the proof. 
\end{proof}

Theorem \ref{thm:invariant generator} shows that the invariant ring of $G_I,G_{II}, G_{III}, G_{IV}$ can be generated by the E-polynomials related to them.
For the details, 
the explicit forms of the generators taken from the E-polynomials are given. 
\begin{align*}
G_{I}  : & \ \varphi_2 = \frac{1}{2} \left( x^2 + y^2 \right), \\
&\ \varphi_8 = \frac{1}{32} \left( 9 x^{8} + 28 x^{6} y^{2} + 70 x^{4} y^{4} + 28 x^{2} y^{6} + 9 y^{8}.
 \right)\\
G_{II}  : & \ \varphi_8 = \frac{1}{24} \left( 5 x^{8} + 70 x^{4} y^{4} + 5 y^{8} \right), \\
& \ \varphi_{24}=\frac{1}{6144}  ( x^{24} + 10626 x^{20} y^{4} + 735471 x^{16} y^{8} + 2704156 x^{12} y^{12}  \\
& \ \ \ \ \ \ \ \ + 735471 x^{8} y^{16} + 10626 x^{4} y^{20} + 1025 y^{24} ) \\
\end{align*}
\begin{align*}
G_{IV}  : & \ \varphi_8 = \frac{1}{2} \left( x^{2} + 3 y^{2}\right), \\
& \ \varphi_{24}=\frac{1}{32}  ( 11 x^{6} + 45 x^{4} y^{2} + 405 x^{2} y^{4} + 243 y^{6}). \\
\end{align*}

\section{Terwilliger Algebra} \label{sec: Terwilliger algebra}

Before continuing the investigation of the Terwilliger algebra, 
the definition group association scheme needs to be defined. 
%Group association schemes is one of association scheme example.
\begin{df}\label{defGroupAssociation}
Let $G$ be a finite group 
%with $|G|=n$ 
and $C_0,C_1,\ldots,C_d$ be the ordering conjugacy classes of $G$. 
Define the relations $R_i(i=0,1,\ldots,d)$ on $G$ by
$$(x,y)\in R_i \Longleftrightarrow yx^{-1}\in C_i.$$
Then $\mathfrak{X}(G)=(G,\lbrace R_i \rbrace_{0\leq i \leq d})$ forms a commutative association scheme of class $d$ called the \textit{group association scheme of $G$}.
\end{df}

We associate the matrix $A_i$ of the relation $R_i$ as 
\begin{equation*}
(A_i)_{x,y} = \begin{cases}
1 & \text{if} \left(x,y\right)\in R_i,\\
0 & \text{otherwise}.
\end{cases}
\end{equation*}
Then,
\begin{equation*}
A_i A_j = \sum_{k=0}^{d}{p_{ij}^k A_k}
\end{equation*}
and  the matrices $A_0,\ldots, A_d$ generate an algebra called a Bose-Mesner algebra.
%and these matrices are basis for the Bose-Mesner algebra $\mathfrak{A}$. 
The intersection numbers $p_{ij}^k$ of the group association scheme $\mathfrak{X}$ are given by 
$$\vert \lbrace \left(x,y\right) \in C_i \times C_j \vert xy=z, z\in C_k\rbrace \vert.$$
For each $i=0,\ldots,d,$ let $E_i^*$ be the diagonal matrices of size $n \times n$ which are defined as follows.
\begin{equation*}
\left( E_i^* \right)_{x,x} = 
	\begin{cases}
	1, & \text{if }  x \in C_i\\
	0, & \text{if }  x \notin C_i
	\end{cases}
	\qquad \left(x\in G \right).
\end{equation*}
Then $\mathfrak{A^*} = \langle E_0^*,\ldots,E_d^* \rangle$ is a commutative algebra called the dual Bose-Mesner algebra.

The intersection numbers provide information about the structure of the Terwilliger algebra. The following relation refers to \cite{terwilliger}.
\begin{equation*}
\begin{array}{c c c c}
E_i^* A_j E_k^* = 0 & \text{iff} & p_{ij}^k=0 & (0\leq i,j,k \leq d)
\end{array}
\end{equation*}

\begin{df}
Let $G$ be a finite group. The Terwilliger algebra $T(G)$ is the sub-algebra of $Mat_G (\mathbb{C})$ generated by $\mathfrak{A}$ and $\mathfrak{A}^*$.
\end{df}

The Terwilliger algebra is noncommutative algebra.
It is also semisimple since it is closed under conjugate-transpose map.
Then,
the investigation of this algebra is undertaken by obtaining its properties,
such as dimension, primitive central idempotent, and structure. 

From \cite{bannaimunemasa},
the bound on the dimension of $T(G)$ turns to be
\[
\vert \lbrace i,j,k \vert p_{ij}^k \neq 0 \rbrace \vert \leq \dim T \leq
\sum_{i=0}^d \frac{|G|}{|C_i|}.
\]
The dimension of $T(G_X)$ is provided. 
\begin{thm} \label{thm:dimension of T}
The dimension of $T(G_I),T(G_{II}),T(G_{III}),$ and $T(G_{IV})$ are given by the followings:
\begin{center}\label{thdim}
	\begin{tabular}{l l l}
	1. & $T(G_I)$ & 64\\
	2. & $T(G_{II})$  & 2808 \\
	3. & $T(G_{III})$ & 300 \\
	4. & $T(G_{IV})$ & 44 \\
	\end{tabular}
\end{center}
\end{thm}
\begin{proof}
The dimension of each case is obtained by determining a basis for each algebra. 
A set $\mathcal{B}$ of linearly independent elements for the set $\lbrace E_i^* A_j E_k^*, E_i^* A_j E_k^* \cdot E_k^* A_l E_m^* \rbrace$ is found.
The computation shows that $\mathcal{B}$ can generates the set 
$\lbrace E_i^* A_j E_k^* \cdot E_k^* A_l E_m^* \cdot E_m^* A_n E_m^p \rbrace$.
This completes the proof.
The details of distribution of the basis elements related to each conjugacy class are given in Appendix \ref{sec:appendix2}.
\end{proof}

Theorem \ref{thm:dimension of T} shos that
$T(G_{IV})$ satisfies the condition 
\[
\dim T = \sum_{i=0}^d \frac{|G_{IV}|}{|C_i|}
\]
where $C_i (i=1, ...)$ are the conjugacy classes of $G_{IV}$.
The readers who want to know from where the dimension of $T(G_X)$ is provided, 
can see Appendix \ref{sec:appendix2}.

After providing the dimension of $T(G_x)$,
the primitive central idempotents need to be obtained. 

We denote by $Z(T)$ the center of $T$.
From \cite{balmaceda},
$Z(T)$ contains block diagonal matrices.
Hence, it can be written as follows.
\[
Z(T) \subseteq \oplus_{i=0}^d Z(E_i^* T E_i^*).
\]
Thus,
to obtain the center of $T$,
it is sufficient to consider the basis elements which are related to $(C_i,C_i)$ position. 

\begin{lem}
The dimensions of the center of $T(G_{I}), T(G_{II}), T(G_{III})$ and $T(G_{IV})$ are the following:
\begin{enumerate}
\item $\dim Z(T(G_I)) = 5$.
\item $\dim Z(T(G_{II})) = 6$.
\item $\dim Z(T(G_{III}) = 3$
\item $\dim Z(T(G_{III}) = 3$
\end{enumerate}
\end{lem}
\begin{proof}
The result is obtained by determining the basis of center, that is the dimension of linear equation system solution $\lbrace x_i y = y x_i \rbrace$ where $y=\sum {c_j b_j}$, $b_j$ and $x_i$ are in the basis of $T$.
\end{proof}

%\begin{df}
%Let $Z(T)$ be center of $T(G)$. The elements $\varepsilon_1,\ldots,\varepsilon_s \in Z(T)$ are called primitive if the followings hold:
%\begin{enumerate}
%\item $\varepsilon_i^2=\varepsilon_i$.
%\item $\varepsilon_i \varepsilon_j = O$ if $i \neq j$.
%\item $\varepsilon_1+\varepsilon_2 + \cdots + \varepsilon_s = I$.
%\end{enumerate}
%\end{df}
%
The degrees of the irreducible complex representation afforded by the primitive central idempotents are provided, that is a set $\lbrace \varepsilon_i \ | \ 1 \leq i \leq s \rbrace$ which satisfies $\varepsilon_i^2 = \varepsilon_i \neq \mathbf{0}, \varepsilon_i \varepsilon_j = \delta_{ij}\varepsilon_i, \sum_{i=1}^2 \varepsilon_i, \text{ and } \varepsilon_i \in Z(T)$.
These are obtained using the method described on \cite{balmaceda}.

Let $e_1, e_2, \ldots, e_s$ be the basis for $\mathcal{Z}(T(G))$.
Therefore,
\[
e_i e_j = \sum r_{ij}^k e_k.
\]
Define a matrix $B$ by
\[
B_i := (r_{ij}^k), 1 \leq i \leq s.
\]
As the matrices $B_i$ are mutually commute,
they can be simultaneously diagonalised by a nonsingular matrix.
Thus, there is a matrix $P$ such that 
\begin{equation} \label{eq:diagonal}
P^{-1} B_i P
\end{equation} 
is a diagonal matrix for $i=1, \ldots, s$.

Let $v_1(i), \ldots, v_s(i)$ be the diagonal entries of (\ref{eq:diagonal}).
Define a matrix $M$ by
\[
M_{ij} := v_i(j).
\]
Then the primitive central idempotents $\varepsilon_1, \ldots, \varepsilon_s$
of $T(G)$ can be obtained by
\[
(\varepsilon_1, \ldots, \varepsilon_s) = (e_1, \ldots, e_s) M^{-1}.
\]
Using the primitive central idempotents, the following result is obtained.

\begin{thm}\label{thpmt}
The degrees of the irreducible complex representations afforded by every idempotents are given below.

\begin{tabular}{c l c c c c c c c c c}
(1)&$T(G_I)$ & $\varepsilon_i$ & $\varepsilon_1$ & $\varepsilon_2$ & $\varepsilon_3$ & $\varepsilon_4$ & $\varepsilon_5$ & & &\\
 &  & \text{deg}$\varepsilon_i$ & 1 & 1 & 2 & 3 & 7 &   & & \\
(2)&$T(G_{II})$ & $\varepsilon_i$ & $\varepsilon_1$ & $\varepsilon_2$ & $\varepsilon_3$ & $\varepsilon_4$ & $\varepsilon_5$ & $\varepsilon_6$ & & \\
 &  & \text{deg}$\varepsilon_i$ & 4 & 8 & 12 & 16 & 24 & 32 &  &  \\
(3)&$T(G_{III})$ & $\varepsilon_i$ & $\varepsilon_1$ & $\varepsilon_2$ & $\varepsilon_3$ & &  & & &\\
 &  & \text{deg}$\varepsilon_i$ & 2 & 10 & 16 &  &  &   & & \\
(3)&$T(G_{III})$ & $\varepsilon_i$ & $\varepsilon_1$ & $\varepsilon_2$ & $\varepsilon_3$ & &  & & &\\
 &  & \text{deg}$\varepsilon_i$ & 2 & 2 & 6 &  &  &   & & \\

\end{tabular}
\end{thm}
\begin{proof}
To determine the degrees of $d_i$ afforded by each $\varepsilon_i$, the fact that $T \varepsilon_i \cong \mathcal{M}_{d_i}(\mathbb{C})$ is used. Thus $d_i^2 = dim T\varepsilon_i$ equals the dimension of the set $\lbrace x_j \varepsilon_i\rbrace$ where $x_j$ are the basis elements of $T$.
\end{proof}

The degrees of primitive idempotents enable us to get the following structure theorem, in which $M_i$ denotes a full matrix algebra over $\mathbb{C}$ of degree $i$.

\begin{cor}[Structure Theorem for $T(G_X)$ ]

\begin{enumerate}

\item $T(G_{I}) \cong \mathcal{M}_1 \oplus \mathcal{M}_1 \oplus \mathcal{M}_2 \oplus \mathcal{M}_3 \oplus \mathcal{M}_7 $.
\item $T(G_{II}) \cong \mathcal{M}_4 \oplus \mathcal{M}_8 \oplus \mathcal{M}_{12} \oplus \mathcal{M}_{16} \oplus \mathcal{M}_{24} \oplus \mathcal{M}_{32} $.
\item $T(G_{III}) \cong \mathcal{M}_2 \oplus \mathcal{M}_{10} \oplus \mathcal{M}_{16}  $.
\item $T(G_{IV}) \cong \mathcal{M}_2 \oplus \mathcal{M}_2 \oplus \mathcal{M}_{6}$.

\end{enumerate}
\end{cor}

\begin{appendices}
\section{Conjugacy Classes} \label{append:conjugacy}

This shows the ordering of the conjugacy classes. 
In this part, 
the two generators for each group are denoted by $\alpha$ and $\beta$. 
The following notations are also used.
\[
\gamma = \alpha \beta, \zeta = \alpha \beta^2, \sigma = \beta \alpha, \mu = \alpha \beta \alpha, \omega = \beta \alpha \beta.
\]

\begin{enumerate}
\item $G_I$

\begin{tabular}{c c c c c c c c}
& $C_0$ & $C_1$ & $C_2$ & $C_3$ & $C_4$ & $C_5$ & $C_6$\\
rep. $C_i$ & $\alpha^2$ & $\gamma^3 \alpha$ & $\gamma^2 \alpha$ & $\gamma^4$ & $\beta \gamma^3$ & $\gamma^3$ & $\gamma $ \\
$|C_i|$ & 1 & 4 & 2 & 1 & 4 & 2 & 2
\end{tabular}

\item $G_{II}$

\begin{tabular}{c c c c c c c c c}
& $C_0$ & $C_1$ & $C_2$ & $C_3$ & $C_4$ & $C_5$ \\
rep. $C_i$ & 
$ \alpha^{2} $ &
$ \zeta^{2} \alpha $ &
$ \zeta \gamma \zeta \sigma \beta\sigma_{2}^2 $ &
$ \zeta \gamma $ &
$ \zeta^{2}$ & 
$ \gamma\zeta^2 \alpha^{2}\sigma \beta $ \\
$|C_i|$ & 1 & 6 & 6 & 6 & 6 & 12  \\
& $C_6$ & $C_7$ & $C_8$ & $C_9$ & $C_{10}$ & $C_{11}$ \\
rep. $C_i$ &
$ \sigma\beta^{2}\zeta $  &
$ \gamma^{3} \beta^{2}\zeta^{2} $ &
$  \zeta^{2} \beta^{2}\zeta $ &
$ \zeta\alpha^{2}\beta\zeta^{2}\gamma\alpha $ &
$ \alpha\beta^{3}\alpha\beta\alpha\beta^{3} $ &
$ \zeta \alpha^2 \sigma \omega^2 \mu \omega^2 $ \\
& $C_{12}$ & $C_{13}$ & $C_{14}$ & $C_{15}$ & $C_{16}$ & $C_{17}$ \\
$|C_i|$ & 12 & 6 & 12 & 12 & 6 & 1  \\
rep. $C_i$& 
$ \sigma_{2} \zeta_{2} $ &
$ \beta_{2} $ &                     
$ \gamma \zeta \gamma^3 \beta \omega^2 \alpha $ &
$ \zeta \sigma \omega^4 $&
$ \zeta \alpha^2 \sigma^2 \beta^3 \gamma^2 \sigma $ &
$ \zeta \alpha^2 \omega^2 \beta^2 \zeta $  \\
& $C_{18}$ & $C_{19}$ & $C_{20}$ & $C_{21}$ & $C_{22}$ & $C_{23}$ \\
$|C_i|$ & 6 & 6 & 8 & 8 & 8 & 1  \\
rep. $C_i$& 
$ \gamma \omega \mu^2 \beta $ & 
$ \gamma \sigma \mu \zeta \omega^2 \sigma $ &
$ \gamma \sigma \gamma^2  \beta^2 \zeta^2 $ &
$ \zeta^3 \sigma^2  $ &
$ \gamma \sigma \gamma^2 \beta \omega $ & 
$ \gamma \sigma \gamma \zeta^2 \sigma^2 $ \\
& $C_{24}$ & $C_{25}$ & $C_{26}$ & $C_{27}$ & $C_{28}$ & $C_{29}$ \\
$|C_i|$ & 1 & 8 & 1 & 6 & 6 & 8  \\
rep. $C_i$& 
$ \zeta^2 \beta \zeta $ &
$ \zeta \omega^3 \sigma $ &
$ (\zeta \beta)^3 $ &
$ \gamma \omega \sigma \mu $ &
$ \mu \omega^2 \zeta \gamma^2 $ &
$ \zeta \mu \beta^3 \gamma^2 $ \\
$|C_i|$ & 8 & 8 & 1 & 6 & 1 & 6  \\
& $C_{30}$ & $C_{31}$ \\
rep. $C_i$& 
$ \alpha_{2}\beta_{2}^{3} $  &                    
$ \sigma^3 $ \\
$|C_i|$ & 8 & 1   \\
\end{tabular}

\item $G_{III}$

\begin{tabular}{c c c c c c c c c}
& $C_0$ & $C_1$ & $C_2$ & $C_3$ & $C_4$ & $C_5$ & $C_6$ \\
rep. $C_i$  &
$ \alpha^{2} $ &
$ \gamma \zeta \mu $ &
$ \gamma \zeta \gamma^2 $ &
$ \zeta \gamma^2 \alpha^{2}\omega$ &
$ \beta^{2} $ &
$ \beta $ & 
$ \beta\alpha\beta^{2}\alpha $ \\
$|C_i|$ & 1 & 4 & 4 & 1 & 4 & 4 & 6 \\
rep. $C_i$  &
$ \gamma \zeta^2 \mu $ &
$ \zeta^2 \alpha $ &
$ \zeta^2 \gamma$ &
$ \sigma^3 $ &
$ \gamma^2 \alpha $ &
$ \sigma^3\beta $ &
$ \sigma \gamma^2  $\\
$|C_i|$ & 1 & 4 & 4 & 1 & 4 & 4 & 6 
\end{tabular}
\item $G_{IV}$

\begin{tabular}{c c c c c c c c c}
& $C_0$ & $C_1$ & $C_2$ & $C_3$ & $C_4$ & $C_5$  \\
rep. $C_i$  &
$ \alpha^{2} $ & 
$ \sigma^3\beta $ &
$ \sigma^2\beta $& 
$ \gamma^2 $ &
$ \alpha\beta $ &
$ \sigma^3 $ \\
$|C_i|$ & 1 & 3 & 3 & 2 & 2 & 1 
\end{tabular}
\end{enumerate}

\section{Appendix B} \label{sec:appendix2}

The following matrices shows the distribution of the basis elements of $T(G_X)$ for $X = \lbrace \text{I, II, III, IV} \rbrace.$
For instance,
for the group $G_I$, 
the entry $1$ in the $(C_1,C_4)$ indicates the basis obtained from the set
\[
\lbrace E_1^* A_j E_k^*, E_i^* A_j E_k^* \cdot E_k^* A_l E_4^* \rbrace.
\]
\[
G_I \ : \ 
\left(\begin{array}{rrrrrrr}
1 & 1 & 1 & 1 & 1 & 1 & 1 \\
1 & 3 & 1 & 1 & 2 & 1 & 1 \\
1 & 1 & 2 & 1 & 1 & 2 & 2 \\
1 & 1 & 1 & 1 & 1 & 1 & 1 \\
1 & 2 & 1 & 1 & 3 & 1 & 1 \\
1 & 1 & 2 & 1 & 1 & 2 & 2 \\
1 & 1 & 2 & 1 & 1 & 2 & 2
\end{array}\right)
\]
\[                                                                                                            
G_{II} \ : \ 
\setlength\arraycolsep{2pt}
\left(\begin{array}{rrrrrrrrrrrrrrrrrrrrrrrrrrrrrrrr}                                                                         
1 & 1 & 1 & 1 & 1 & 1 & 1 & 1 & 1 & 1 & 1 & 1 & 1 & 1 & 1 & 1 & 1 & 1 & 1 & 1 & 1 & 1 & 1 & 1 & 1 & 1 & 1 & 1 & 1 & 1 & 1 & 1 
\\                                                                                                                            
1 & 3 & 3 & 3 & 3 & 3 & 3 & 3 & 3 & 3 & 3 & 1 & 3 & 3 & 2 & 2 & 2 & 1 & 1 & 2 & 1 & 3 & 3 & 2 & 2 & 2 & 1 & 3 & 1 & 3 & 2 & 1 
\\                                                                                                                            
1 & 3 & 3 & 3 & 3 & 3 & 3 & 3 & 3 & 3 & 3 & 1 & 3 & 3 & 2 & 2 & 2 & 1 & 1 & 2 & 1 & 3 & 3 & 2 & 2 & 2 & 1 & 3 & 1 & 3 & 2 & 1 
\\                                                                                                                            
1 & 3 & 3 & 3 & 3 & 3 & 3 & 3 & 3 & 3 & 3 & 1 & 3 & 3 & 2 & 2 & 2 & 1 & 1 & 2 & 1 & 3 & 3 & 2 & 2 & 2 & 1 & 3 & 1 & 3 & 2 & 1 
\\                                                                                                                            
1 & 3 & 3 & 3 & 3 & 3 & 3 & 3 & 3 & 3 & 3 & 1 & 3 & 3 & 2 & 2 & 2 & 1 & 1 & 2 & 1 & 3 & 3 & 2 & 2 & 2 & 1 & 3 & 1 & 3 & 2 & 1 
\\                                                                                                                            
1 & 3 & 3 & 3 & 3 & 5 & 5 & 3 & 5 & 5 & 3 & 1 & 3 & 3 & 3 & 3 & 3 & 1 & 1 & 3 & 1 & 3 & 3 & 3 & 3 & 3 & 1 & 3 & 1 & 3 & 3 & 1 
\\                                                                                                                            
1 & 3 & 3 & 3 & 3 & 5 & 5 & 3 & 5 & 5 & 3 & 1 & 3 & 3 & 3 & 3 & 3 & 1 & 1 & 3 & 1 & 3 & 3 & 3 & 3 & 3 & 1 & 3 & 1 & 3 & 3 & 1 
\\                                                                                                                            
1 & 3 & 3 & 3 & 3 & 3 & 3 & 3 & 3 & 3 & 3 & 1 & 3 & 3 & 2 & 2 & 2 & 1 & 1 & 2 & 1 & 3 & 3 & 2 & 2 & 2 & 1 & 3 & 1 & 3 & 2 & 1 
\\                                                                                                                            
1 & 3 & 3 & 3 & 3 & 5 & 5 & 3 & 5 & 5 & 3 & 1 & 3 & 3 & 3 & 3 & 3 & 1 & 1 & 3 & 1 & 3 & 3 & 3 & 3 & 3 & 1 & 3 & 1 & 3 & 3 & 1 
\\                                                                                                                            
1 & 3 & 3 & 3 & 3 & 5 & 5 & 3 & 5 & 5 & 3 & 1 & 3 & 3 & 3 & 3 & 3 & 1 & 1 & 3 & 1 & 3 & 3 & 3 & 3 & 3 & 1 & 3 & 1 & 3 & 3 & 1 
\\
1 & 3 & 3 & 3 & 3 & 3 & 3 & 3 & 3 & 3 & 3 & 1 & 3 & 3 & 2 & 2 & 2 & 1 & 1 & 2 & 1 & 3 & 3 & 2 & 2 & 2 & 1 & 3 & 1 & 3 & 2 & 1 
\\
1 & 1 & 1 & 1 & 1 & 1 & 1 & 1 & 1 & 1 & 1 & 1 & 1 & 1 & 1 & 1 & 1 & 1 & 1 & 1 & 1 & 1 & 1 & 1 & 1 & 1 & 1 & 1 & 1 & 1 & 1 & 1 
\\
1 & 3 & 3 & 3 & 3 & 3 & 3 & 3 & 3 & 3 & 3 & 1 & 3 & 3 & 2 & 2 & 2 & 1 & 1 & 2 & 1 & 3 & 3 & 2 & 2 & 2 & 1 & 3 & 1 & 3 & 2 & 1 
\\
1 & 3 & 3 & 3 & 3 & 3 & 3 & 3 & 3 & 3 & 3 & 1 & 3 & 3 & 2 & 2 & 2 & 1 & 1 & 2 & 1 & 3 & 3 & 2 & 2 & 2 & 1 & 3 & 1 & 3 & 2 & 1 
\\
1 & 2 & 2 & 2 & 2 & 3 & 3 & 2 & 3 & 3 & 2 & 1 & 2 & 2 & 4 & 4 & 4 & 1 & 1 & 4 & 1 & 2 & 2 & 4 & 4 & 4 & 1 & 2 & 1 & 2 & 4 & 1 
\\
1 & 2 & 2 & 2 & 2 & 3 & 3 & 2 & 3 & 3 & 2 & 1 & 2 & 2 & 4 & 4 & 4 & 1 & 1 & 4 & 1 & 2 & 2 & 4 & 4 & 4 & 1 & 2 & 1 & 2 & 4 & 1 
\\
1 & 2 & 2 & 2 & 2 & 3 & 3 & 2 & 3 & 3 & 2 & 1 & 2 & 2 & 4 & 4 & 4 & 1 & 1 & 4 & 1 & 2 & 2 & 4 & 4 & 4 & 1 & 2 & 1 & 2 & 4 & 1 
\\
1 & 1 & 1 & 1 & 1 & 1 & 1 & 1 & 1 & 1 & 1 & 1 & 1 & 1 & 1 & 1 & 1 & 1 & 1 & 1 & 1 & 1 & 1 & 1 & 1 & 1 & 1 & 1 & 1 & 1 & 1 & 1 
\\
1 & 1 & 1 & 1 & 1 & 1 & 1 & 1 & 1 & 1 & 1 & 1 & 1 & 1 & 1 & 1 & 1 & 1 & 1 & 1 & 1 & 1 & 1 & 1 & 1 & 1 & 1 & 1 & 1 & 1 & 1 & 1 
\\
1 & 2 & 2 & 2 & 2 & 3 & 3 & 2 & 3 & 3 & 2 & 1 & 2 & 2 & 4 & 4 & 4 & 1 & 1 & 4 & 1 & 2 & 2 & 4 & 4 & 4 & 1 & 2 & 1 & 2 & 4 & 1 
\\
1 & 1 & 1 & 1 & 1 & 1 & 1 & 1 & 1 & 1 & 1 & 1 & 1 & 1 & 1 & 1 & 1 & 1 & 1 & 1 & 1 & 1 & 1 & 1 & 1 & 1 & 1 & 1 & 1 & 1 & 1 & 1 
\\
1 & 3 & 3 & 3 & 3 & 3 & 3 & 3 & 3 & 3 & 3 & 1 & 3 & 3 & 2 & 2 & 2 & 1 & 1 & 2 & 1 & 3 & 3 & 2 & 2 & 2 & 1 & 3 & 1 & 3 & 2 & 1 
\\
1 & 3 & 3 & 3 & 3 & 3 & 3 & 3 & 3 & 3 & 3 & 1 & 3 & 3 & 2 & 2 & 2 & 1 & 1 & 2 & 1 & 3 & 3 & 2 & 2 & 2 & 1 & 3 & 1 & 3 & 2 & 1 
\\
1 & 2 & 2 & 2 & 2 & 3 & 3 & 2 & 3 & 3 & 2 & 1 & 2 & 2 & 4 & 4 & 4 & 1 & 1 & 4 & 1 & 2 & 2 & 4 & 4 & 4 & 1 & 2 & 1 & 2 & 4 & 1 
\\
1 & 2 & 2 & 2 & 2 & 3 & 3 & 2 & 3 & 3 & 2 & 1 & 2 & 2 & 4 & 4 & 4 & 1 & 1 & 4 & 1 & 2 & 2 & 4 & 4 & 4 & 1 & 2 & 1 & 2 & 4 & 1 
\\
1 & 2 & 2 & 2 & 2 & 3 & 3 & 2 & 3 & 3 & 2 & 1 & 2 & 2 & 4 & 4 & 4 & 1 & 1 & 4 & 1 & 2 & 2 & 4 & 4 & 4 & 1 & 2 & 1 & 2 & 4 & 1 
\\
1 & 1 & 1 & 1 & 1 & 1 & 1 & 1 & 1 & 1 & 1 & 1 & 1 & 1 & 1 & 1 & 1 & 1 & 1 & 1 & 1 & 1 & 1 & 1 & 1 & 1 & 1 & 1 & 1 & 1 & 1 & 1 
\\
1 & 3 & 3 & 3 & 3 & 3 & 3 & 3 & 3 & 3 & 3 & 1 & 3 & 3 & 2 & 2 & 2 & 1 & 1 & 2 & 1 & 3 & 3 & 2 & 2 & 2 & 1 & 3 & 1 & 3 & 2 & 1 
\\
1 & 1 & 1 & 1 & 1 & 1 & 1 & 1 & 1 & 1 & 1 & 1 & 1 & 1 & 1 & 1 & 1 & 1 & 1 & 1 & 1 & 1 & 1 & 1 & 1 & 1 & 1 & 1 & 1 & 1 & 1 & 1 
\\
1 & 3 & 3 & 3 & 3 & 3 & 3 & 3 & 3 & 3 & 3 & 1 & 3 & 3 & 2 & 2 & 2 & 1 & 1 & 2 & 1 & 3 & 3 & 2 & 2 & 2 & 1 & 3 & 1 & 3 & 2 & 1 
\\
1 & 2 & 2 & 2 & 2 & 3 & 3 & 2 & 3 & 3 & 2 & 1 & 2 & 2 & 4 & 4 & 4 & 1 & 1 & 4 & 1 & 2 & 2 & 4 & 4 & 4 & 1 & 2 & 1 & 2 & 4 & 1 
\\
1 & 1 & 1 & 1 & 1 & 1 & 1 & 1 & 1 & 1 & 1 & 1 & 1 & 1 & 1 & 1 & 1 & 1 & 1 & 1 & 1 & 1 & 1 & 1 & 1 & 1 & 1 & 1 & 1 & 1 & 1 & 1
\end{array}\right)
\]

\[
G_{III} \ : \
\left(\begin{array}{rrrrrrrrrrrrrr}
1 & 1 & 1 & 1 & 1 & 1 & 1 & 1 & 1 & 1 & 1 & 1 & 1 & 1 \\
1 & 2 & 2 & 1 & 2 & 2 & 2 & 1 & 2 & 2 & 1 & 2 & 2 & 2 \\
1 & 2 & 2 & 1 & 2 & 2 & 2 & 1 & 2 & 2 & 1 & 2 & 2 & 2 \\
1 & 1 & 1 & 1 & 1 & 1 & 1 & 1 & 1 & 1 & 1 & 1 & 1 & 1 \\
1 & 2 & 2 & 1 & 2 & 2 & 2 & 1 & 2 & 2 & 1 & 2 & 2 & 2 \\
1 & 2 & 2 & 1 & 2 & 2 & 2 & 1 & 2 & 2 & 1 & 2 & 2 & 2 \\
1 & 2 & 2 & 1 & 2 & 2 & 3 & 1 & 2 & 2 & 1 & 2 & 2 & 3 \\
1 & 1 & 1 & 1 & 1 & 1 & 1 & 1 & 1 & 1 & 1 & 1 & 1 & 1 \\
1 & 2 & 2 & 1 & 2 & 2 & 2 & 1 & 2 & 2 & 1 & 2 & 2 & 2 \\
1 & 2 & 2 & 1 & 2 & 2 & 2 & 1 & 2 & 2 & 1 & 2 & 2 & 2 \\
1 & 1 & 1 & 1 & 1 & 1 & 1 & 1 & 1 & 1 & 1 & 1 & 1 & 1 \\
1 & 2 & 2 & 1 & 2 & 2 & 2 & 1 & 2 & 2 & 1 & 2 & 2 & 2 \\
1 & 2 & 2 & 1 & 2 & 2 & 2 & 1 & 2 & 2 & 1 & 2 & 2 & 2 \\
1 & 2 & 2 & 1 & 2 & 2 & 3 & 1 & 2 & 2 & 1 & 2 & 2 & 3
\end{array}\right) 
\]

\[
G_{IV} \ : \
\left(\begin{array}{rrrrrr}
1 & 1 & 1 & 1 & 1 & 1 \\
1 & 2 & 2 & 1 & 1 & 1 \\
1 & 2 & 2 & 1 & 1 & 1 \\
1 & 1 & 1 & 2 & 2 & 1 \\
1 & 1 & 1 & 2 & 2 & 1 \\
1 & 1 & 1 & 1 & 1 & 1
\end{array}\right)
\]

\end{appendices}

\begin{thebibliography}{99}

\bibitem{balmaceda}
Balmaceda, J.M.P., Oura, M.,
The Terwilliger algebras of the group association schemes of $S_5$ and $A_5$,
Kyushu J. Math. 48 (1994), no. 2, 221-231. 

\bibitem{balmacedareyes}
Balmaceda, J.M.P., Reyes, A.C.L.,
Terwilliger algebras of certain group association schemes,
Journal of the Math. Soc. Philippines. 44 (2021), no. 2, 19-28.

%\bibitem{bannai} 
%Bannai, E., Ito, T.,
%Algebraic Combinatorics I: Association Schemes,
%Benjamin/Cummings, California, 1984.

\bibitem{bannaimunemasa}
Bannai, E., Munemasa, A.,
The Terwilliger algebras of group association schemes,
Kyushu J. Math. 49 (1995), no. 1, 93-102. 

%\bibitem{bogaertsdukes}
%Bogaerts, M., Dukes, P.,
%Semidefinitie Programming for Permutation Codes,
%Discrete Mathematics 326 (2014) 34-43.

%\bibitem{bosmaetal}
%Bosma, W., Cannon, J., Playoust, C.,
%The Magma algebra system I: The user language,
%J. Symbolic Comput. 24 (1997), no. 3-4, 235-265. 

%\bibitem{BCN}
%Brouwer, A. E., Cohen, A. M., Neumaier, A., 
%Distance-regular graphs,
%Ergebnisse der Mathematik und ihrer Grenzgebiete (3), 18. Springer-Verlag, Berlin, 1989.

\bibitem{fernadezmiklavic}
Fern\'andez, B, Miklavi\'c, S.,
On bipartite graphs with exactly one irreducible $T$-module with endpoint 1, which is thin,
European J. Combin. 97  (2021), 1-15.


\bibitem{HamidTerwilliger}
Hamid, N.,
Terwilliger Algebras of Some Group Association Schemes,
Math. J. Okayama Univ. 61 (2019), 199-204.

\bibitem{Hamid}
Hamid, N.,
Note on E-polynomials Associated to $\mathbb{Z}_4$-codes,
Nihonkai Mathematical Journal 30 (2), 31-40.

\bibitem{hanaki}
Hanaki, A. Modular Terwilliger Algebras of Association Schemes. Graphs and Combinatorics 37, 1521–1529 (2021). https://doi.org/10.1007/s00373-021-02363-0

%\bibitem{ito}
%T. Ito,
%Tridiagonal pairs and q-Onsager algebras(Japanese),
%S\={u}gaku 65 (2013), no. 1, 69-92. 

%\bibitem{martintanaka}
%W. Martin, H. Tanaka,
%Commutative association schemes,
%European J. Combin. 30 (2009), no. 6, 1497-1525. 

\bibitem{MiezakiOura}
T. Miezaki, M. Oura,
On Eisenstein polynomials and zeta polynomials II
Int. J. Number Theory 16, No. 1, 2020 (February), 2-7-218.

\bibitem{ourapoly}
Oura, M.,
Eisenstein polynomials essociated to binary codes,
Int. J. Number THeory 5 (1009), no. 4, 645-640.

\bibitem{sagemath}
SageMath, the Sage Mathematics Software System (Version 6.4.1),
The Sage Developers, 2014, http://www.sagemath.org.

\bibitem{terwilliger}
P. Terwilliger, 
The subconstituent algebra of an association scheme I,
J. Algebraic Combin. 1 (1992), no. 4, 363-388,
II. 2 (1993), no. 1, 73-103,
III.  2 (1993), no. 2, 177-210.


\end{thebibliography}
\end{document}